%% file: Qn.tex
\documentclass[a4paper,11pt]{amsart}

\usepackage{amsmath}
\usepackage{amssymb}
\usepackage{graphicx}
\usepackage{bbold}
\usepackage{mathrsfs}
\usepackage{booktabs}
\usepackage{multirow}
\usepackage{dcolumn}
\usepackage{units} 
\usepackage{hhline} 
\usepackage{colortbl}

\theoremstyle{plain}
\newtheorem{theorem}{Theorem}[section]
\newtheorem{proposition}[theorem]{Proposition}
\newtheorem{corollary}[theorem]{Corollary}
\newtheorem{lemma}[theorem]{Lemma}
\newtheorem{fact}[theorem]{Fact}

\theoremstyle{definition}
\newtheorem{definition}[theorem]{Definition}
\newtheorem{remark}[theorem]{Remark}
\newtheorem{example}[theorem]{Example}

\DeclareMathOperator{\conv}{conv}
\DeclareMathOperator{\cone}{cone}
\DeclareMathOperator{\tr}{tr}
\DeclareMathOperator{\N0}{\mathbb{0}}
\newcommand{\abs}[1]{{\left\lvert{#1}\right\rvert}}
\newcommand{\pAbs}[2]{{\left\lvert{#1}\right\rvert_{#2}}}
\newcommand{\close}[1]{\overline{#1}}
\newcommand{\RR}{\mathbb{R}}
\newcommand{\ZZ}{\mathbb{Z}}
\newcommand{\Mone}[1]{{{\mathbb 1}_{#1}}}%
\newcommand{\Tp}{{\mspace{-1mu}\scriptscriptstyle\top\mspace{-1mu}}}
\newcommand{\cplmt}{\complement}
\newcommand{\idn}[1]{\imath_{#1}}
\newcommand{\id}{\imath}
\newcommand{\One}{\mathbf{1}}
\newcommand{\Zero}{\mathbf{0}}

\newcommand{\SzMsymb}{{\mathbb S}\mspace{-3.2mu}{}^{\scriptscriptstyle0}}
\newcommand{\SzM}[1]{\SzMsymb_{#1}}

\newcommand{\MM}{\mathbb{M}}
\newcommand{\lt}{\left}
\newcommand{\rt}{\right}
\newcommand{\Mtxv}[1]{{\lt(\begin{smallmatrix}#1\end{smallmatrix}\rt)}}
\newcommand{\pair}[2]{#1\!\!\!\nearrow\!\!#2}
\newcommand{\Cm}[1]{C^{\text{\tiny$#1$}}}
\newcommand{\path}{{-}}
\newcommand{\stbox}[2]{\text{\tiny\parbox{#1}{\centering\baselineskip=0.5 \baselineskip #2}}}
\newcommand{\ubtxt}[3]{\underbrace{#1}_{\stbox{#2}{#3}}}
\newcommand{\M}[1]{\mathbb{#1}}
\newcommand{\eps}{\varepsilon}

\begin{document}

\title[On a Class of Metrics Related to Graph Layout Problems]{On a Class of Metrics Related to\\ Graph Layout Problems}

\author[Letchford]{Adam N.~Letchford}%
\address{Department of Management Science, Lancaster University,
  United Kingdom. E-mail: {\tt a.n.letchford@lancaster.ac.uk}}%
\author[Reinelt]{Gerhard Reinelt}%
\address{Gerhard Reinelt, Institute of Computer Science, University of
  Heidelberg, Germany. E-mail: {\tt
    Gerhard.Reinelt@informatik.uni-heidelberg.de}}%
\author[Seitz]{Hanna Seitz}%
\address{Hanna Seitz, Institute of Computer Science, University of
  Heidelberg, Germany. E-mail: {\tt
    Hanna.Seitz@informatik.uni-heidelberg.de}}%
\author[Theis]{Dirk Oliver Theis}%
\address{D\'epartement de Math\'ematique, Universit\'e Libre de
  Bruxelles, Belgium. E-mail: {\tt theis@uni-heidelberg.de}}%

\date{May 2010}

\begin{abstract}
We examine the metrics that arise when a finite set of points is embedded in
the real line, in such a way that the distance between each pair of points is
at least $1$. These metrics are closely related to some other known metrics in
the literature, and also to a class of combinatorial optimization problems
known as graph layout problems. We prove several results about the structure
of these metrics. In particular, it is shown that their convex hull is not
closed in general. We then show that certain linear inequalities define facets
of the closure of the convex hull. Finally, we characterise the unbounded edges
of the convex hull and of its closure.
\\*[1mm]
\end{abstract}

\maketitle

\input{Introduction}

\input{Literatur1}

\input{Literatur2}

\input{Facets}

\input{UnboundedEdges}

\section{Concluding Remarks} \label{se:end}

The $\RR$-embeddable $1$-separated metrics are a natural and fascinating class
of metrics, which are also of some practical importance due to their
connection with graph layout problems. We have established some fundamental
properties of such metrics, and also initiated a study of their convex hull
and its closure. 

There are several possible avenues for future research. First, one could
search for 
new valid or facet-defining inequalities. Second, one could study the
complexity of 
the separation problems associated with various families of inequalities,
which would 
be essential if one wished to use the inequalities within a cutting-plane
algorithm. 
Third, it would be interesting to know whether the {\it bounded}\/ edges of
the convex hull, or its closure, have a simple combinatorial interpretation.\\
\\
{\bf Acknowledgement:} The first author was supported by the Engineering and
Physical 
Sciences Research Council under grant {\tt EP/D072662/1}. The third author was
supported by Deutsche Forschungsgemeinschaft (DFG) within grant RE~776/9-1 and
by the Communaut\'e  fran\c caise de Belgique under Actions de Recherche
Concert\'ees. 

\input{Appendix}

\bibliographystyle{plain}
\bibliography{Qn}

\end{document}

%% file: Introduction.tex
\section{Introduction} \label{sec:intro}

For a given positive integer $n$, let $[n]$ denote $\{1,\dots,n\}$. A
{\it metric}\/ on $[n]$ is a mapping $d \colon [n] \times [n] \to \RR_+$
which satisfies the following three conditions:
\begin{itemize}
\item $d(i,j) = d(j,i)$ for all $\{i,j\} \subset [n]$,
\item $d(i,k) + d(j,k) \ge d(i,j)$ for all ordered triples $(i,j,k) \subset [n]$,
\item $d(i,j) = 0$ if and only if $i = j$.
\end{itemize}
Metrics are a special case of {\it semimetrics}, which are obtained by dropping
`and only if' from the third condition. There is a huge literature on metrics and
semimetrics; see for example \cite{DL97}. 
The inequalities in the second condition are the well-known {\it triangle inequalities}.

In this paper we study the metrics $d$ on $[n]$ that arise when $n$ points are
embedded in the real line, in such a way that the distance between each pair of
points is at least $1$. More formally, we require that $d$ satisfies the following
two properties:
\begin{itemize}
\item there exist real numbers $r_1, \dots, r_n$ such that
$d(i,j) = \abs{r_i-r_j}$ for all $\{i,j\} \subset [n]$;
\item $d(i,j) \ge 1$ for all $\{i,j\} \subset [n]$.
\end{itemize}
We remark that one could easily replace the value $1$ with some arbitrary constant
$\epsilon > 0$; the results in this paper would remain essentially unchanged.

We call the metrics in question `$\RR$-embeddable $1$-separated' metrics.
We believe that these metrics are a natural object of study, and of interest in
their own right. We have, however, two specific motives for studying
them. First, 
they are closely related to certain well-known metrics that have appeared
in the literature. Second, they are also closely related to an important class
of combinatorial optimization problems, known as \emph{graph layout problems}.

As well as studying the metrics themselves, we also study their convex hull. It
turns out that the convex hull is not always closed, which leads us to study
also the closure of the convex hull. Among other things, we characterise some
of the $(n-1)$-dimensional faces (i.$\,$e., facets) of the closure, and some of
the $1$-dimensional faces (i.$\,$e., edges) of both the convex hull and its
closure.

The structure of the paper is as follows. In Section~\ref{se:literature}, we 
review some of the relevant literature on metrics and graph layout
problems. In Section~\ref{se:structure}, we present various results concerned
with the structure of the metrics and their convex hull. Next, in Section~\ref{se:facets}, we present some inequalities that define facets of the
closure of the convex hull. In Section~\ref{se:edges}, we give a combinatorial
characterisation of the unbounded edges of the convex hull and of its closure.
Finally, some concluding remarks are given in Section~\ref{se:end}.

We close this section with a word on notation. 
To study convex geometric properties, we view metrics as points in
a vector space $\SzM{n}$.  In our notation, $\SzM{n}$ will be either
the vector space of all symmetric functions $[n]\times[n] \to \RR$ or
the vector space of all real symmetric $(n\times n)$-matrices whose
diagonal entries are zero, and we will switch freely between them.
For the latter, the inner product is defined as usual by
\begin{equation*}
  A\bullet B := \tr(A^\Tp B) = \sum_{k=1}^n\sum_{l=1}^n A_{k,l} B_{k,l}.
\end{equation*}
We understand a metric both as a function and a matrix, and we
will switch between the two concepts without further mentioning.

By $S(n)$ we denote the set of all permutations of $[n]$.  We
occasionally view $S(n)$ as a subset of $\RR^d$ by identifying the
permutation $\pi$ with the point $(\pi(1),\dots,\pi(n))^\Tp$.
Furthermore we let $\idn{n} := (1,\dots,n)$ the identity permutation
in $S(n)$.  We omit the index $n$ when no confusion can arise.  
$\One$~is a column vector of appropriate length consisting of ones.
Similarly $\Zero$ is a vector whose entries are all zero.  If
appropriate, we will use a subscript $\One_k$, $\Zero_k$ to identify
the length of the vectors.  The symbol $\N0$ denotes an all-zeros
matrix not necessarily square, and we also use it to say ``this part
of the matrix consists of zeros only.'' By $\Mone{n}$ we denote the
square matrix of order~$n$ whose $(k,l)$-entry is $1$ if $k\ne l$ and
$0$ otherwise.  As above we will omit the index $n$ when appropriate.
We denote by $\cplmt U$ the complement of the set $U$.


%% file: Literatur1.tex
\section{Literature Review} \label{se:literature}

In this section, we review some of the relevant literature. We cover
related semimetrics in Subsection \ref{sub:lit1} and graph layout
problems in Subsection \ref{sub:lit2}.  To facilitate reading we have
summarized all matrix sets discussed in Table~\ref{table:Summary}.

\begin{table}
  \begin{center}
    \begin{tabular}{lll} 
      \toprule  
      CUT$_n$& $\ell_1$-embeddable semimetrics (cut cone) \\
      \addlinespace[1.0ex]
      HYP$_n$& hypermetrics, see~\eqref{eq:hypermetric}\\
      \addlinespace[1.0ex]
      NEG$_n$ & negative-type cone, see~\eqref{eq:negtype}\\
      \addlinespace[1.0ex]
      $M_n^{L2}$ &  $\ell_2$-embeddable semimetrics\\
      \addlinespace[1.0ex]
      $M_n^{R}$ & $\RR$-embeddable semimetrics\\
      \addlinespace[1.0ex]
      $M_n^{R1}$ &  $\RR$-embeddable $1$-separated metrics\\
      \addlinespace[1.0ex]
      $Q_n$ & convex hull of $M_n^{R1}$ \\
      \addlinespace[1.0ex]
      $\close{Q_n}$ & closure of $Q_n$\\
      \addlinespace[1.0ex]
      $P_n$ & permutation metrics polytope, see~\eqref{eq:def-Pn}\\
      \addlinespace[1.0ex]
      \bottomrule
    \end{tabular}
  \end{center}
  \label{table:Summary}
  \caption{Sets of matrices}
\end{table}

\subsection{Some related semimetrics} \label{sub:lit1}

The following four classes of semimetrics on $[n]$, which are closely
related to the $\RR$-embeddable $1$-separated metrics, have been extensively
studied in the literature (see \cite{DL97} for a detailed survey):
\begin{itemize}
\item The {\it $\ell_1$-embeddable}\/ semimetrics, i.$\,$e., those for which
there exist a positive integer $m$ and points $x_1, \dots, x_n \in \RR^m$ such
that $d(i,j) = \pAbs{x_i-x_j}{1} := \sum_{k=1}^m \abs{x_{ik}-x_{jk}}$ for all
$\{i,j\} \subset [n]$.
\item The {\it $\ell_2$-embeddable}\/ semimetrics, which are defined as in the
$\ell_1$ case, except that
$d(i,j) = \pAbs{x_i-x_j}{2} := \sqrt{\sum_{k=1}^m (x_{ik}-x_{jk})^2}$.
\item The {\it $\RR$-embeddable}\/ semimetrics, which are the special case
of $\ell_1$- (or $\ell_2$-) embeddable semimetrics obtained when $m=1$.
\item The {\it hypermetrics}, which are semimetrics that satisfy
the following {\it hypermetric}\/ inequalities \cite{D61}:
\begin{equation}\label{eq:hypermetric}
  \sum_{\{i,j\} \subset [n]} b_ib_j d(i,j) \le 0 \qquad
  (\forall b \in \ZZ^n: \sum_{i=1}^n b_i = 1).
\end{equation}
\end{itemize}
It is known \cite{A80} that the set of $\ell_1$-embeddable semimetrics on $[n]$
is a polyhedral cone in $\RR^{\binom{n}{2}}$. In fact, it is nothing but the
well-known {\it cut cone}, denoted by CUT$_n$. The set of all hypermetrics on
$[n]$, called the {\it hypermetric cone}\/ and denoted by HYP$_n$, is also
polyhedral \cite{DGL93}.

We will let $M_n^{L2}$ and $M_n^R$ denote the set of $\ell_2$- and
$\RR$-embeddable semimetrics, respectively. It is known that
$M_n^{L2}$ and $M_n^R$ are not convex (unless $n$ is small), and that
the convex hull of $M_n^{L2}$ and $M_n^R$ is CUT$_n$. It is also
known~\cite{S35} that a symmetric function $d$ lies in $M_n^{L2}$ if
and only if $d^2$ (i.$\,$e., the symmetric function obtained by
squaring each value) lies in the so-called {\it negative-type cone}.
The negative-type cone, denoted by NEG$_n$, is the (non-polyhedral)
cone defined by the following {\it negative-type}\/ inequalities:
\begin{equation}\label{eq:negtype}
\sum_{\{i,j\} \subset [n]} b_ib_j d(i,j) \le 0 \qquad
(\forall b \in \RR^n: \sum_{i=1}^n b_i = 0).
\end{equation}
The structure of $M_n^R$ and related sets is studied in \cite{BD92}.

In recent years, there has been a stream of papers on so-called
{\it negative-type}\/ semimetrics (also known as $\ell_2^2$-semimetrics)
\cite{ALN07,ALN08,CGR08,KV05,KR06,L05}. These are simply semimetrics that lie
in NEG$_n$. They have been used to derive approximation algorithms for various
combinatorial optimisation problems, including the graph layout problems that
we mention in the next subsection.

The following inclusions are known: $M_n^R \subset M_n^{L2} \subset$
CUT$_n \subset$ HYP$_n \subset$ NEG$_n$. Denoting the set of all 
$\RR$-embeddable $1$-separated metrics by $M_n^{R1}$, we obtain from their definition 
$M_n^{R1} \subset M_n^R$. We will explore the relationship between
$M_n^{R1}$, $M_n^R$ and CUT$_n$ further in Subsection \ref{sub:structure1}.

\subsection{Graph layout problems} \label{sub:lit2}

Given a graph $G=(V,E)$, with $V = [n]$, a {\it layout}\/ is simply a
permutation of $[n]$. If we view a layout $\pi \in S(n)$ as a placing of
the vertices on points $1,\ldots,n$ along the real line, the quantity
$|\pi(i) - \pi(j)|$ corresponds to the Euclidean distance between vertices
$i$ and $j$. Several important combinatorial optimization problems,
collectively known as {\it graph layout problems}, call for a layout
minimising a function of these distances (see the survey \cite{DPS02}). For
example, in the {\it Minimum Linear Arrangement Problem}\/ (MinLA), the
objective is to minimize $\sum_{\{i,j\} \in E} |\pi(i) - \pi(j)|$. In the
{\it Bandwidth Problem}, the objective is to minimise
$\max_{\{i,j\} \in E} |\pi(i) - \pi(j)|$.

Now, let $d(i,j)$ for $\{i,j\} \subset [n]$ be a decision variable,
representing the quantity $|\pi(i) - \pi(j)|$. It has been observed by
several authors that interesting relaxations of graph layout problems
can be formed by deriving valid linear inequalities that are satisfied
by all feasible symmetric functions $d$.  To our knowledge, the first
paper of this kind was \cite{LV95}, which presented the following {\it
  star}\/ inequalities:
\begin{equation}  \label{eq:star}
\sum_{j \in S} d(i,j) \ge \lfloor (\abs S+1)^2/4 \rfloor.
\end{equation}
Here, $i \in [n]$ and $S \subset [n] \setminus \{i\}$ is such that every node
in $S$ is adjacent to $i$.

Apparently independently, Even \emph{et al.}\ \cite{ENRS00} defined the
so-called {\it spreading metrics}. These are metrics that satisfy the following
{\it spreading}\/ inequalities:
\begin{equation}  \label{eq:spreading}
\sum_{j \in S} d(i,j) \ge \abs{S}(\abs{S}+2)/4
\qquad (\forall i \in [n], \forall S \subseteq [n] \setminus \{i\}).
\end{equation}
Note that the spreading inequalities are more general than the star inequalities,
but have a slightly weaker right-hand side when $n$ is odd. Spreading metrics were
used in \cite{ENRS00,RR05} to derive approximation algorithms for various graph
layout problems.

In \cite{CHKR08,FL07}, it was noted that one can get a tighter relaxation of
graph layout problems by requiring the spreading metrics to lie in the
negative-type cone NEG$_n$. The authors called the resulting metrics
{\it $\ell_2^2$-spreading}\/ metrics.

A natural way to derive further valid linear inequalities for graph
layout problems is to study the following \textit{permutation metrics polytope:}
\begin{equation}\label{eq:def-Pn}
P_n = \conv \Bigl\{ 
d \Bigm|
\exists\pi\in S(n):\; d(i,j)=\abs{\pi(i)-\pi(j)}\;\forall \{i,j\} \subset [n]
\Bigr\}.
\end{equation}
Surprisingly, this was not
done until very recently \cite{AL09}. In \cite{AL09}, it is shown that
$P_n$ is of dimension $\binom{n}{2}-1$ and that its affine hull is
defined by the equation $\sum_{\{i,j\} \subset [n]} d(i,j) = \binom{n+1}{3}$.
It is also shown that the following four classes of inequalities define
facets of $P_n$ under mild conditions:
\begin{itemize}
\item {\it pure hypermetric}\/ inequalities, which are simply the hypermetric
inequalities (\ref{eq:hypermetric}) for which $b \in \{0, \pm 1 \}^n$;
\item {\it strengthened pure negative-type}\/ inequalities, which are like the
negative-type inequalities (\ref{eq:negtype}) for which $b \in \{0, \pm 1 \}^n$,
except that the right-hand side is increased from $0$ to
$\frac{1}{2} \sum_{i \in [n]} |b_i|$;
\item {\it clique}\/ inequalities, which take the form
\begin{equation} \label{eq:clique}
\sum_{\{i,j\} \subset S} d(i,j) \ge \binom{|S|+1}{3},
\end{equation}
where $S \subset [n]$ satisfies $2 \le |S| < n$;
\item {\it strengthened star}\/ inequalities, which take the form
\begin{equation} \label{eq:strongstar}
(|S|-1) \sum_{i \in S} d(r,i) - \sum_{\{i,j\} \subset S} d(i,j) \ge
\left\lfloor (|S|+1)^2(|S|-1)/12 \right\rfloor,
\end{equation}
where $r \in V$ and $S \subseteq V \setminus \{r\}$ with $|S| \ge 2$.
\end{itemize}
It is pointed out in the same paper that each star inequality (\ref{eq:star})
with $|S| \ge 2$ is dominated by a clique inequality (\ref{eq:clique}) and a
strengthened star inequality (\ref{eq:strongstar}). Therefore, very few of
the star inequalities define facets of $P_n$.

Finally, we mention that some more valid inequalities were presented
recently by Caprara \emph{et al.}\ \cite{CLS09}. Some of them were
proved to define facets of the \emph{dominant}\/ of $P_n$, though not
of $P_n$ itself.

We will establish an interesting connection between $M_n^{R1}$, CUT$_n$
and $P_n$ in Subsection \ref{sub:structure2}.


%% file: Literatur2.tex
\section{On $M_n^{R1}$ and its Convex Hull} \label{se:structure}

\subsection{On $M_n^{R1}$ and related sets} \label{sub:structure1}

We now study $M_n^{R1}$ and its relationship with $M_n^R$, $P_n$ and
CUT$_n$. We will find it helpful to recall the definition of a {\it cut metric}:

\begin{definition}
  For a set $U\subset[n]$, we let $d_U$ be the metric which assigns to
  two points on different sides of the bipartition $U,\cplmt U$ of
  $[n]$ a value of $1$ and to points on the same side a value of $0$.
\end{definition}
We will say that the set $U$ {\it induces}\/ the associated cut
metric.
In other words, if we let $D_{k,l}(x) := \abs{x_k-x_l}$ for every
vector $x\in\RR^n$ (and identify, as promised, functions and
matrices), then $d_U = D(\chi^U)$.  
With this notation, CUT$_n$ is the convex cone with apex $0$ in
$\SzM{n}$ generated by the points $d_U$, i.$\,$e.,
\begin{equation*}
  \text{CUT}_n := \cone \Bigl\{ d_U \Bigm| d_U \;\text{is the cut metric for}\; U \subset [n] \Bigr\}.
\end{equation*} 
It is known \cite{BM86} that each cut metric defines an extreme ray of
CUT$_n$.

We will also need the following notation.
For a given permutation $\pi \in S(n)$, let $N_\pi$ be the set of
$x\in\RR^n$ which satisfy $x_{\pi(i)} \le x_{\pi(i+1)}$ for $i = 1,
\ldots, n-1$.  Now let $M(\pi)$ denote the set of metrics $d$ for
which there exists an $x\in N_\pi$ with $d=D(x)$.
Also, for a given $\pi$ and for $k=1, \ldots, n-1$, we emphasize that
$D(\chi^{\pi^{-1}([k])})$ is the cut metric induced by the set $U =
\{\pi^{-1}(1), \ldots, \pi^{-1}(k)\}$. (So, for example, if $n=4$ and
$\pi = \{2,3,1,4\}$, then $D(\chi^{\pi^{-1}([2])})$ is the cut metric induced
by the set $\{2,3\}$.)

We have the following lemma:
\begin{lemma}\label{le:sum1}%
  $M(\pi)$ is a polyhedral cone of dimension $n-1$ defined by the
  $n-1$ cut metrics $D(\chi^{\pi^{-1}([1]))}, \ldots, D(\chi^{\pi^{-1}([n-1])})$.
\end{lemma}
\begin{proof}
  Let $d^* \in M(\pi)$ and let $x_1, \ldots, x_n$ be the corresponding
  points in $\RR$. One can check that:
  \[
  d^* = \sum_{k=1}^{n-1} (x_{k+1} - x_{k}) D(\chi^{\pi^{-1}([k])}).
  \]
  From the definition of $M(\pi)$, we have $x_{k+1} - x_{k} \ge 0$ for
  $k = 1, \ldots, n-1$. Thus, $d^*$ is a conical combination of the
  $n-1$ cut metrics mentioned. This shows that $M(\pi)$ is contained
  in the cone mentioned.  The reverse direction is similar.
\end{proof}

This enables us to describe the structure of $M_n^R$:
\begin{proposition}\label{pr:union1}
  $M_n^R$ is the union of $n!/2$ polyhedral cones, each of dimension $n-1$.
\end{proposition}
We define the {\it antipodal} permutation of $\pi\in S(n)$ by
\begin{equation*}
  \pi^- := (n+1)\cdot\One - \pi.
\end{equation*}
This is the permutation obtained by reversing $\pi$.  A swift
computation shows that $D(\pi)=D(\pi^-)$.
\begin{proof}
  From the definitions, we have $M_n^R = \bigcup_{\pi \in S(n)}
  M(\pi)$.  From the above lemma, the set $M(\pi)$ is a polyhedral
  cone of dimension $n-1$.  Now, note that, for any $\pi \in S(n)$, we
  have $M(\pi) = M(\pi^-)$.  Thus, the union can be taken over $n!/2$
  permutations, instead of over all permutations.
\end{proof}

We note in passing that every cut metric belongs to $M(\pi)$ for some
$\pi \in S(n)$. This explains the well-known fact, mentioned in
Subsection \ref{sub:lit1}, that the convex hull of $M_n^R$ is equal to
CUT$_n$.

Now, we adapt these results to the case of $M_n^{R1}$.  We define
$M^1(\pi)$ similar to $M(\pi)$: we denote by $M^1(\pi)$ the set of all
metrics $d$ which are of the form $D(x)$ for an $x\in\RR^n$ which
satisfies $x_{\pi(i)} +1 \le x_{\pi(i+1)}$ for $i = 1, \ldots, n-1$.

Note that the $D(\pi)$ are nothing but the metrics associated with
feasible layouts, which by a result in \cite{AL09} are the extreme
points of $P_n$.  Note also that the sets $M^1(\pi)$ are disjoint.

We have the following lemma:
\begin{lemma} \label{le:sum2}
  $M^1(\pi)$ is the Minkowski sum of the point $D(\pi)$ and the cone
  $M(\pi)$:
  \begin{equation*}
    M^1(\pi) = D(\pi) + D(N_\pi).
  \end{equation*}
\end{lemma}
\begin{proof}
  This can be proven in the same way as Lemma \ref{le:sum1}. The only difference
  is that we decompose $d^* \in M^1(\pi)$ as:
  \[
  d^* = D(\pi) + \sum_{k=1}^{n-1} (r_{k+1} - r_{k} - 1) D(\chi^{\pi^{-1}([k])}),
  \]
  and note that $r_{k+1} - r_{k} - 1 \ge 0$ for $k = 1, \ldots, n-1$.
\end{proof}

We can now derive an analog of Proposition \ref{pr:union1}:
\begin{proposition} \label{pr:union2} $M_n^{R1}$ is the union of
  $n!/2$ disjoint translated polyhedral cones, each of dimension
  $n-1$.
\end{proposition}
\begin{proof}
  From the definitions, we have $M_n^{R1} = \bigcup_{\pi \in S(n)}
  M^1(\pi)$.  From Lemmas \ref{le:sum1} and \ref{le:sum2}, each set
  $M^1(\pi)$ is a translated polyhedral cone of dimension $n-1$. As in
  the proof of Proposition \ref{pr:union1}, the union can be taken
  over only $n!/2$ permutations.
\end{proof}

\subsection{On the convex hull of $M_n^{R1}$ and related sets}
\label{sub:structure2}

We now turn our attention to the convex hull of $M_n^{R1}$, which we denote
by $Q_n$. To give some intuition, we present in Fig.~\ref{fig:Q_3} drawings
of $M_n^{R1}$ and $Q_3$ from three different angles. (Of course, the drawing
is truncated, since $Q_3$ is unbounded.) The three co-ordinates represent
$d(1,2)$, $d(1,3)$ and $d(2,3)$. The three coloured regions represent the three
disjoint subsets of $M_3^{R1}$ mentioned in Proposition \ref{pr:union2}.

One can see that $Q_3$ is a three-dimensional polyhedron, with one bounded facet,
six unbounded facets, three bounded edges and six unbounded edges.

\begin{figure}
  \includegraphics[scale=.29]{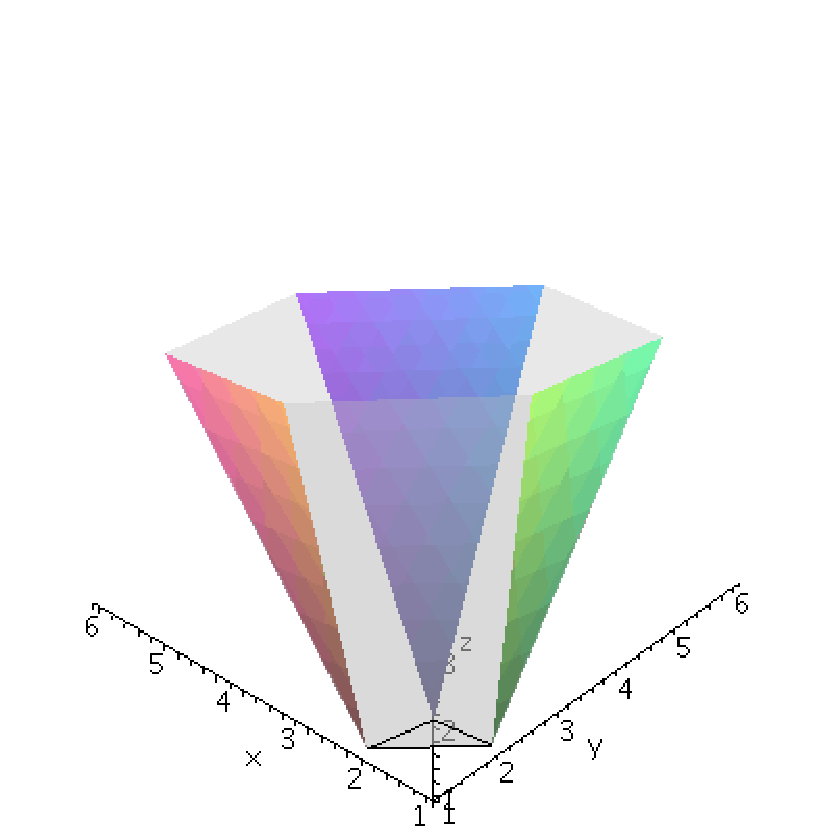}
  \includegraphics[scale=.29]{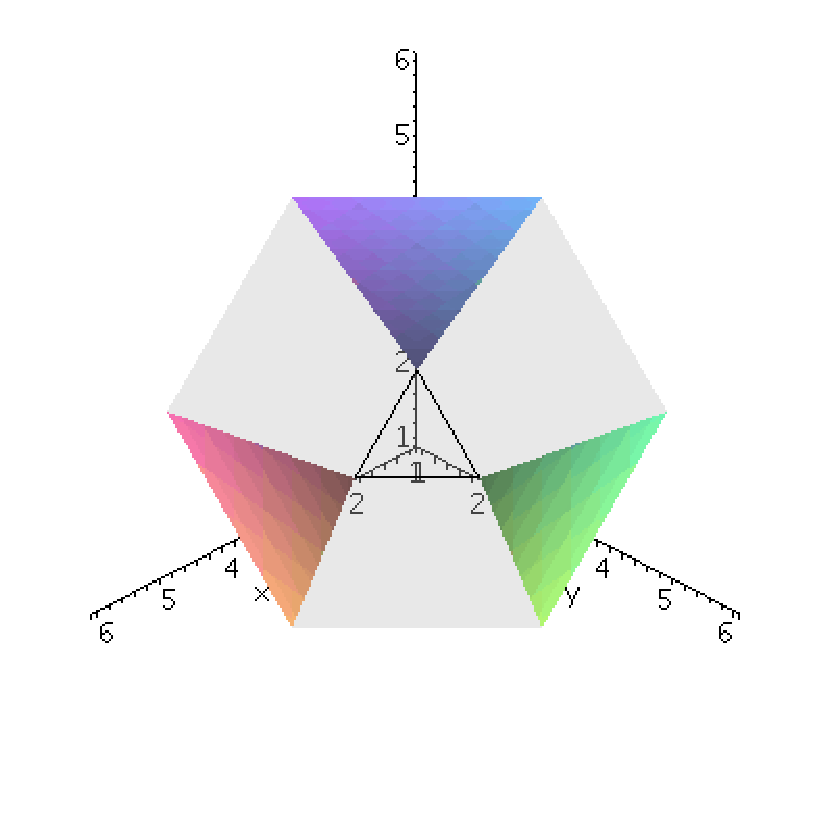}
  \includegraphics[scale=.29]{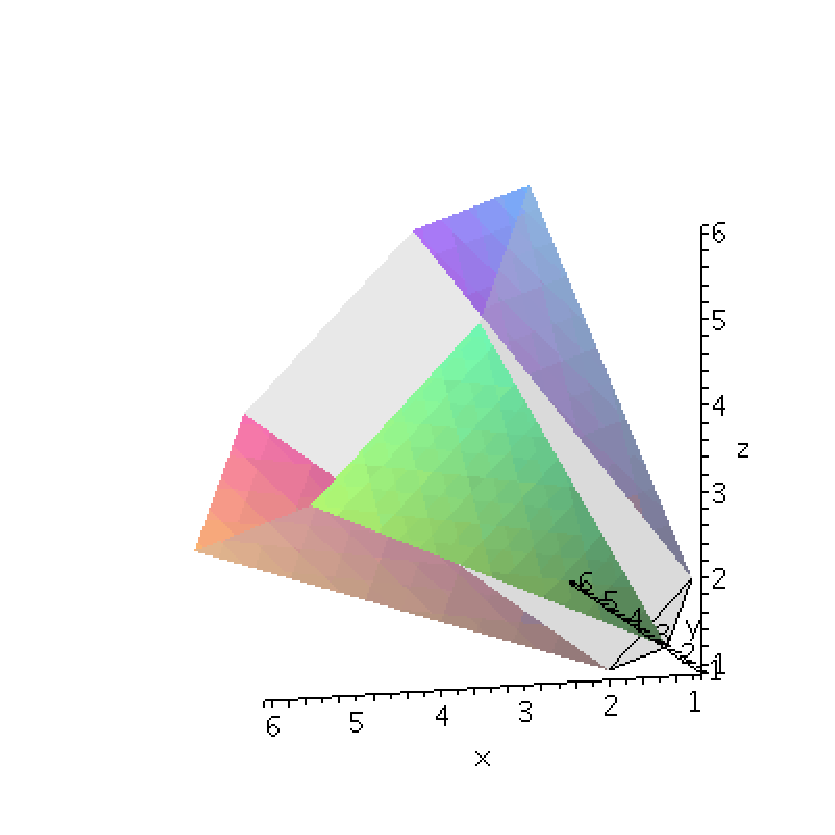}
  \caption{The convex set $Q_3$}
  \label{fig:Q_3}
\end{figure}

For $n \le 3$, $Q_n$ is closed (and therefore a polyhedron). We will show
in Section \ref{se:edges}, however, that $Q_n$ is not closed for
$n \ge 4$.
Therefore, we are led to look at the closure of $Q_n$, which we denote by
$\close{Q_n}$.

Our next result shows that there is a close connection between the polyhedron
$\close{Q_n}$, the polytope $P_n$, and the cone CUT$_n$:
\begin{proposition}
$\close{Q_n}$ is the Minkowski sum of $P_n$ and CUT$_n$.
\end{proposition}
\begin{proof}
We use the same notation as in the previous subsection. By definition, every
point in $M_n^{R1}$ belongs to $M^1(\pi)$ for some $\pi \in S(n)$. From
Lemma \ref{le:sum2}, every point in $M^1(\pi)$ is the sum of the point $D(\pi)$
and a point in the cut cone CUT$_n$. Moreover, the point $D(\pi)$ is an
extreme point of~$P_n$. Thus, every point in $M_n^{R1}$ is the sum of an
extreme point of~$P_n$ and a point in CUT$_n$. Since $\close{Q_n}$ is the
closure of the convex hull of $M_n^{R1}$, it must be contained in the Minkowski
sum of $P_n$ and CUT$_n$. The reverse direction is proved similarly, noting
that every cut metric is of the form $D(\chi^{\pi^{-1}([k])})$ for some $\pi \in S(n)$
and some $k \in [n-1]$.
\end{proof}

This immediately implies the following result:
\begin{corollary}
$\close{Q_n}$ is full-dimensional (i.$\,$e., of dimension $\binom{n}{2}$).
\end{corollary}

We also have the following result:
\begin{proposition} \label{pr:boundedfacet}
$P_n$ is the unique bounded facet of $\close{Q_n}$.
\end{proposition}
\begin{proof}
As mentioned in the previous section, all points in $P_n$ satisfy the equation
$\sum_{\{i,j\} \subset [n]} d(i,j) = \binom{n+1}{3}$. Moreover, every point in
CUT$_n$ satisfies $\sum_{\{i,j\} \subset [n]} d(i,j) > 0$. Since $\close{Q_n}$
is the Minkowski sum of $P_n$ and CUT$_n$, it follows that the inequality
$\sum_{\{i,j\} \subset [n]} d(i,j) \ge \binom{n+1}{3}$ is valid for
$\close{Q_n}$ and that $P_n$ is the face of $\close{Q_n}$ exposed by this
inequality. Since $\close{Q_n}$ and $P_n$ are of dimension $\binom{n}{2}$ and
$\binom{n}{2}-1$, respectively, $P_n$ is a facet of $\close{Q_n}$. It must be
the unique bounded facet, since all extreme points of $\close{Q_n}$ are in
$P_n$.
\end{proof}

In the next section, we will explore the connection between $\close{Q_n}$,
$P_n$ and CUT$_n$ in more detail. To close this section, we make an observation
about how the individual `pieces' of $M_n^{R1}$, called the $M^1(\pi)$ in the
previous subsection, are positioned within $\close{Q_n}$:

\begin{proposition}\label{prop:n-1Dimensional}
For any $\pi \in S(n)$, the set $M^1(\pi)$ is an $(n-1)$-dimensional face
of $\close{Q_n}$.
\end{proposition}
\begin{proof}
By definition, $\close{Q_n}$ satisfies all triangle inequalities. Now, without
loss of generality, suppose that $\pi$ is the identity permutation. Every
point in $M^1(\pi)$ satisfies all of the following triangle inequalities at
equality:
\[
d(i,j) + d(j,k) \ge d(i,k) \qquad (\forall 1 \le i < j < k \le n).
\]
Moreover, no other point in $M_n^{R1}$ does so. Thus, $M^1(\pi)$ is a
face of $\close{Q_n}$. It was shown to be $(n-1)$-dimensional in the
previous subsection. 
\end{proof}


%% file: Facets.tex
\section{Inequalities Defining Facets of $\close{Q_n}$} \label{se:facets}

In this section, we study linear inequalities that define {\it facets}\/ of
$\close{Q_n}$, i.$\,$e., faces of dimension $\binom{n}{2}-1$. Subsection
\ref{sub:facets1} presents some general results about such inequalities,
whereas Subsection \ref{sub:facets2} lists some specific inequalities.

\subsection{General results on facet-defining inequalities}
\label{sub:facets1}

In this subsection, we prove a structural result about inequalities that
define facets of $\close{Q_n}$, and show how this can be used to construct
facets of $\close{Q_n}$ in a mechanical way from facets of either $P_n$ or
CUT$_n$.

We will need the following definition, taken from \cite{AL09}:
\begin{definition}[Amaral \& Letchford, 2009]
Let $\alpha^Td \ge \beta$ be a linear inequality, where
$\alpha, d \in \RR^{\binom{n}{2}}$. The inequality is said to be `canonical'
if:
\begin{equation} \label{eq:canonical}
\min_{\emptyset \ne S \subset [n]} \sum_{i \in S} \sum_{[n] \setminus S} \alpha_{ij} = 0.
\end{equation}
\end{definition}

By definition, an inequality $\alpha^Td \ge 0$ defines a proper face of CUT$_n$
if and only if it is canonical. In \cite{AL09}, it is shown that every facet of
$P_n$ is defined by a canonical inequality. The following lemma is the analogous
result for $\close{Q_n}$:
\begin{lemma}
Every unbounded facet of $\close{Q_n}$ is defined by a canonical inequality.
\end{lemma}
\begin{proof}
Suppose that the inequality $\alpha^Td \ge \beta$ defines an unbounded facet of
$\close{Q_n}$. Since $\close{Q_n}$ is the Minkowski sum of $P_n$ and CUT$_n$, the
inequality must be valid for CUT$_n$. Therefore, the left-hand side of
(\ref{eq:canonical}) must be non-negative. Moreover, since the inequality defines
an unbounded facet, there must be at least one extreme ray of CUT$_n$ satisfying
$\alpha^Td = 0$. Therefore the left-hand side of (\ref{eq:canonical}) cannot be
positive.
\end{proof}
We remind the reader that only one facet of $\close{Q_n}$ is bounded
(Proposition \ref{pr:boundedfacet}).

Now, we show how to derive facets of $\close{Q_n}$ from facets of $P_n$:
\begin{proposition} \label{pr:fromPn}
Let $F$ be any facet of $P_n$, and let $\alpha^Td \ge \beta$ be the canonical
inequality that defines it. This inequality defines a facet of $\close{Q_n}$
as well.
\end{proposition}
\begin{proof}
The fact that the inequality is valid for $\close{Q_n}$ follows from the fact
that $\close{Q_n}$ is the Minkowski sum of $P_n$ and CUT$_n$. Now, since $F$
is a facet of $P_n$, there exist $\binom{n}{2}-1$ affinely-independent vertices
of $P_n$ that satisfy the inequality at equality. Moreover, since the inequality
is canonical, there exists at least one extreme ray of CUT$_n$ that satisfies
$\alpha^Td = 0$. Since $\close{Q_n}$ is the Minkowski sum of $P_n$ and CUT$_n$,
there exist $\binom{n}{2}$ affinely-independent points in $\close{Q_n}$ that
satisfy the inequality at equality. Thus, the inequality defines a facet of
$\close{Q_n}$.
\end{proof}

Now, we show how to derive facets of $\close{Q_n}$ from facets of CUT$_n$:
\begin{proposition} \label{pr:fromCUTn}
Let $\alpha^Td \ge 0$ define a facet of CUT$_n$, and let $\beta$ be the
minimum of $\alpha^Td$ over all $d \in P_n$. Then the inequality
$\alpha^Td \ge \beta$ define a facet of $\close{Q_n}$.
\end{proposition}
\begin{proof}
As before, the fact that the inequality $\alpha^Td \ge \beta$ is valid for
$\close{Q_n}$ follows from the fact that $\close{Q_n}$ is the Minkowski sum of
$P_n$ and CUT$_n$. Now, since the inequality $\alpha^Td \ge 0$ defines a facet
of CUT$_n$, there exist $\binom{n}{2}-1$ linearly-independent extreme rays of
CUT$_n$ that satisfy $\alpha^Td = 0$. Moreover, from the definition of $\beta$,
there exists at least one extreme point of $P_n$ that satisfies
$\alpha^Td = \beta$. Since $\close{Q_n}$ is the Minkowski sum of $P_n$ and
CUT$_n$, there exist $\binom{n}{2}$ affinely-independent points in $\close{Q_n}$
that satisfy $\alpha^Td = \beta$. Thus, the inequality $\alpha^Td \ge \beta$
defines a facet of $\close{Q_n}$.
\end{proof}

\subsection{Some specific facet-defining inequalities} \label{sub:facets2}

The results in the previous subsection enable one to derive a wide variety
of facets of $\close{Q_n}$. In this subsection, we briefly examine some
specific valid inequalities; namely, the inequalities mentioned in
\cite{AL09}.

First, we deal with the clique and pure hypermetric inequalities:
\begin{proposition}
The clique inequalities (\ref{eq:clique}) define facets of $\close{Q_n}$
for all $S \subseteq [n]$ with $|S| \ge 2$.
\end{proposition}
\begin{proof}
It was shown in \cite{AL09} that the clique inequalities define facets of
$P_n$ when $S$ is a proper subset of $[n]$. In this case, the inequalities are
canonical and so, by Proposition \ref{pr:fromPn}, they define facets of
$\close{Q_n}$ as well. The case $S = [n]$ is covered in the proof of
Proposition \ref{pr:boundedfacet}.
\end{proof}

\begin{proposition}
All pure hypermetric inequalities define facets of $\close{Q_n}$.
\end{proposition}
\begin{proof}
It was shown in \cite{BM86} that all pure hypermetric inequalities define
facets of CUT$_n$. It was also shown in \cite{AL09} that every pure hypermetric
inequality is satisfed at equality by at least one extreme point of $P_n$.
The result then follows from Proposition \ref{pr:fromCUTn}. 
\end{proof}

As for the strengthened pure negative-type and strengthened star inequalities,
it was shown in \cite{AL09} that they define facets of $P_n$ under certain
conditions. Since they are canonical, they define facets of $\close{Q_n}$ under
the same conditions. In fact, using the same proof technique used in \cite{AL09},
one can show the following two results:
\begin{proposition}
All strengthened pure negative-type inequalities define facets of $\close{Q_n}$.
\end{proposition}

\begin{proposition}
Strengthened star inequalities define facets of $\close{Q_n}$ if and only if
$|S| \neq 4$.
\end{proposition}
We omit the proofs, for the sake of brevity.


%% file: UnboundedEdges.tex
\section{Unbounded Edges of $Q_n$ and $\close{Q_n}$} \label{se:edges}

\subsection{Unbounded edges of $Q_n$}
We now investigate \textsl{how} the polyhedral cones $M^1(\pi) =
D(\pi)+D(N_\pi)$ as subsets of $Q_n$.  In Fig.~\ref{fig:Q_3}, it can
be seen that in the case $n=3$, the three cones are faces of $Q_3$
(recall that $Q_3$ is a polyhedron, which means that we can safely
speak of faces).  In the following proposition, we show that this is
the case for all $n$, and we also characterize the extremal half-lines
of $Q_n$.  This will be useful in comparing $Q_n$ with its closure: We
will characterize the unbounded edges issuing from each vertex for the
polyhedron $\close{Q_n}=P_n+$CUT$_n$ in the following subsection.

We are dealing with an unbounded convex set of which we do not know whether it
is closed or not.  (In fact, we will show that $Q_n$ is 
almost never closed).  For this purpose, we supply the following fact for easy
reference. 

\begin{fact}\label{fact:detailsForQn6b}%
  For $k=1,\dots,m$ let $K_k$ be a (closed) polyhedral cone with apex $x_k$.
  Suppose that the $K_k$ are pairwise disjoint and define $S :=
  \biguplus_{k=1}^m 
  K_k$.  Let $x,y$ be vectors such that $x+\RR_+y$ is an extremal subset of
  $\conv(S)$.  It then follows that there exists a $\lambda_0\in\RR_+$ and a
  $k$ such 
  that $x+\lambda y\in K_k$ for all $\lambda\ge\lambda_0$.  Since $x+\RR_+y$
  is extremal, this implies that there exists a $\lambda_1\in\RR_+$ such that
  $x_k = x+\lambda_1 y$ and $x_k+\RR_+ y=\{x+\lambda y\mid
  \lambda\ge\lambda_1\}$ is an extreme ray of the polyhedral cone $K_k$. 
\end{fact}

\begin{definition}
We say that a permutation $\pi$ and a non-empty set $U\subsetneq[n]$ are {\it
  incident,} if $U = \{\pi^{-1}(1),\dots,\pi^{-1}(k)\}$, where $k:=\abs
  U$. 
\end{definition}

\begin{proposition}\label{cor:x-hlines-in-Qn}\mbox{}%
  \renewcommand{\theenumi}{\roman{enumi}}
  \begin{enumerate}
  \item\label{cor:x-hlines-in-Qn:a} For every $\pi \in S(n)$, each
    edge of the cone $D(\pi)+D(N_\pi)$ is an exposed subset of $Q_n$.
  \item\label{cor:x-hlines-in-Qn:b} The unbounded one dimensional
    extremal sets of $Q_n$ are exactly the defining half-lines.  In
    other words, every half-line $X+\RR_+Y$ which is an extremal
    subset of $Q_n$ is of the form $D(\pi)+\RR_+D(\chi^U)$ for a
    $\pi\in S(n)$ and a set $U$ incident to $\pi$.
    In particular, for every vertex $D(\pi)$ of $Q_n$, the unbounded
    one-dimensional extremal subsets of $Q_n$ containing $D(\pi)$ are
    in bijection with the non-empty proper subsets of $[n]$ incident
    to $\pi$.  Thus there are precisely $n-1$ of them.
\end{enumerate}
\end{proposition}
\begin{proof}
  \textit{\ref{cor:x-hlines-in-Qn:a}.} By symmetry it is sufficient to
  treat the case $\pi=\id:=(1,\dots,n)^\Tp$, the identity permutation.
  Consider the matrix
  \begin{equation*}
    C := 
    \Mtxv{
      0 &1  & &       & &   &-1\\
      1 &0  &1&       & &\N0&  \\
      ~ &1  & &       & &   &  \\
      ~ &   & &\ddots & &   &  \\
      ~ &   & &       & &1  &\\
      ~ &\N0& &       &1&0  &1\\
      -1&   & &       & &1  &0
    }
    \in \SzM{n}.
  \end{equation*}
  It is easy to see that the minimum over all $C\bullet D(\pi)$, $\pi\in
  S(n)$, is attained only in 
  $\pi=\id,\id^-$ with the value $0$.  Moreover, for any non-empty proper
  subset $U$ of $[n]$, we have 
  $C\bullet D(\chi^U) = 0$ if $U$ is incident to $\id$ and $C\bullet D(\chi^U)
  > 0$ otherwise.  Hence, we have that $D(\id) + D(N_\id)$ is equal to the set
  of all points in $Q_n$ which satisfy the valid inequality $C\bullet X\ge 0$ with equality.
  Out of this matrix $C$ we will now construct a matrix $C'$ and a
  right hand side such that only some of the subsets incident to $\id$
  fulfill the inequality with equality.  To do so let $U_0$ be a
  subsets of $[n]$ incident to $\id$.  If, for each $U \subset [n]$
  incident to $\id$ but different from $U_0$, we increase the matrix
  entries $C_{\max U,\max U +1}$ and $C_{\max U+1,\max U}$ by one, we
  obtain an inequality $C'\bullet X\ge 0$ which is valid for $Q_n$ and
  such that the set of all points of $Q_n$ which are satisfied with
  equality is precisely the edge of $D(\id)+D(N_\id)$ generated by the
  half-lines $D(\id) + \RR_+D(\chi^{U_0})$.
  
  {\it \ref{cor:x-hlines-in-Qn:b}.} That the defining half lines are extremal
  has just been proved in $\ref{cor:x-hlines-in-Qn:a}$. 
  The converse statement follows from Fact~\ref{fact:detailsForQn6b} and
  the fact that the extreme points of $Q_n$ are precisely the vertices of
  $P_n$, which are of the form $D(\pi)$, for $\pi\in S(n)$.
\end{proof}

\subsection{Unbounded edges in $\close{Q_n}$}\label{sec:unbdEdges}

We have just identified some unbounded edges of $\close{Q_n}=P_n+$CUT$_n$
starting at a particular vertex $D(\pi)$ of this polyhedron.
We now set off to characterize {\sl all} unbounded edges of $\close{Q_n}$.
Clearly, the unbounded edges are of the form $D(\pi)+\RR_+ D(\chi^U)$, but
not all these half-lines are edges. 
For a permutation $\pi$ and a non-empty subset $U\subsetneq[n]$, we say that
$D(\pi)+\RR_+ D(\chi^U)$ is the half-line {\it defined by the pair $\pair{\pi}{U}$}.
In this section, we characterize the pairs $\pair{\pi}{U}$ which have
the property that the half-lines they define are edges.  For this, we
make the following definition.

\begin{definition}
  Let $\pi$ be a permutation, and let $U$ be a subset of $[n]$.  We
  say that $U$ is 
  \textit{almost incident} to $\pi$, if there exists
  a $k\in[n-1]$ such that $U=\pi^{-1}([k-1]\cup\{k+1\})$.
\end{definition}

We can now state our theorem.

\begin{theorem}\label{thm:conjectureAboutExraysProven}
  For all $n\ge 3$, the unbounded edges of $\close{Q_n}$ are precisely
  the half-lines defined by those pairs $\pair{\pi}{U}$, for which
  neither $U$ nor $\cplmt U$ is almost incident to $\pi$.
\end{theorem}

From Theorem~\ref{thm:conjectureAboutExraysProven}, we have the
following consequences.

\begin{corollary}
  For $n\ge 4$, the number of unbounded edges issuing from a vertex of
  $\close{Q_n}=P_n+C_n$ is $2^{n-1}-n$.
\end{corollary}

\begin{corollary}
  For $n\ge 4$, the extremal half-lines containing an extreme point of
  $Q_n$ are a proper subset of the unbounded edges issuing from the
  same vertex of $\close{Q_n}$.
\end{corollary}
\begin{proof}
  We have $n-1 < 2^{n-1}-n$ if $n\ge 4$.
\end{proof}

\begin{corollary}
  The convex set $Q_n$ is closed if and only if $n \leq 3$.
\end{corollary}

Major parts of the proof of the above stated theorem work in an inductive fashion by
reducing to the case when $n\in\{3,4,5,6\}$.  We will present the cases $n=3$
and $n=4$ as examples, which also helps motivating the definitions we require for the proof.

We will switch to a more ``visual'' notation of the subsets of $[n]$ by
identifying a set $U$ with a ``word'' of length $n$ over $\{0,1\}$ having a
$1$ in the $j$th position iff $j\in U$ --- it is just the row-vector $(\chi^U)^\Tp$.

\begin{example}[Unbounded edges of $\close{Q_3}$]\label{exp:unbdEdges:n-equal-3}
  We deal with the case $n=3$ ``visually'' by regarding Fig.~\ref{fig:Q_3}.
  There are two edges starting 
  at each vertex.  In fact, with some computation, it can be seen that the
  unbounded edges containing $D(\id)$ are
  \begin{align*}
    M\Mtxv{1\\2\\3}+\RR_+ M\Mtxv{1\\0\\0}
    &=
    \Mtxv{0&1&2\\1&0&1\\2&1&0} + \RR_+\Mtxv{0&1&1\\1&0&0\\1&0&0}
    \text{, \quad and}
    \\
    M\Mtxv{1\\2\\3}+\RR_+ M\Mtxv{1\\1\\0}
    &= \Mtxv{0&1&2\\1&0&1\\2&1&0} + \RR_+\Mtxv{0&0&1\\0&0&1\\1&1&0}
    \text{; \quad while}
    \\
    M\Mtxv{1\\2\\3}+\RR_+ M\Mtxv{1\\0\\1}
    &=
    \Mtxv{0&1&2\\1&0&1\\2&1&0} + \RR_+\Mtxv{0&1&0\\1&0&1\\0&1&0}
  \end{align*}
  is not an edge.  This agrees with Proposition~\ref{cor:x-hlines-in-Qn},
  because the sets $100$ and 
  $110$ are incident to $\id$, while $101$ and $010$ are not.
  Moreover, the set $101$ is almost incident to $\id$ and $010$ is its complement.  Thus,
  Theorem~\ref{thm:conjectureAboutExraysProven} is true for the special case
  when $\pi = \id$.  For the 
  other permutations, the easiest thing to do is to use symmetry.  We describe
  this in the next remark. 
\end{example}

\begin{remark}\label{rem:symm-Qn-and-inciovrdg}
 \renewcommand{\theenumi}{\roman{enumi}}
  For every $\sigma,\pi\in S(n)$ and $U\subset[n]$ we have the following.
  \begin{enumerate}
  \item Due to symmetry the pair $\pair{\pi}{U}$ defines an edge
    of $\close{Q_n}$ if and 
    only if the pair $\pair{\pi\circ \sigma}{\sigma^{-1}(U)}$ defines an edge
    of $\close{Q_n}$. 
  \item $U$ is incident to $\pi$ if and only if $\sigma^{-1}(U)$ is incident
  to $\pi\circ\sigma$. 
  \item $U$ is almost incident to a permutation $\pi$ if and only if
    $\sigma^{-1}(U)$ is almost incident $\pi\circ\sigma$.
  \item $\cplmt U$ is almost incident to a permutation $\pi$ if and only if
    $U$ is almost incident to $\pi^-$.
  \end{enumerate}
\end{remark}
\begin{proof} Can be checked using the definitions of $\pair{\pi}{U}$ and
    $U$ beeing incident respectively almost incident of $\pi$.   
\end{proof}

We now give the first general result as a step towards the proof of
Theorem~\ref{thm:conjectureAboutExraysProven}.

\begin{lemma}\label{lem:theyAreNoExtremeRays}
  If $\pi\in S(n)$ and $U\subset[n]$ is almost incident $\pi$, then the half-line
  $D(\pi)+\RR_+D(\chi^U)$ defined by the pair $\pair{\pi}{U}$ is not an edge
  of $\close{Q_n}$. 
\end{lemma}
\begin{proof}
  By the above remarks on symmetry, it is sufficient to prove the claim for
  the identical permutation $\id\in S(n)$.
  Consider a $k\in[n-1]$, and let $\pi' := \lt<k,k+1\rt>$ be the transposition
  exchanging $k$ and $k+1$, and let $U := [k-1]\cup\{k+1\}$.  Then a little
  computation shows that $D(\chi^U)$ can be written as a  
  conic combination of vectors defining rays issuing from $D(\id)$ as follows:
  \begin{equation*}
    D(\chi^U)  =  D(\chi^{[k]})  +  \bigl( D(\pi')  -  D(\id) \bigr).
  \end{equation*}
  Hence $D(\id)+\RR_+D(\chi^U)$ is not an edge.
\end{proof}

Note that by applying Remark~\ref{rem:symm-Qn-and-inciovrdg}, the
Lemma~\ref{lem:theyAreNoExtremeRays} 
implies that if $\cplmt U$ is almost incident $\pi$, then the pair
$\pair{\pi}{\cplmt U}$ does not 
define an edge of $\close{Q_n}$.

Before we proceed, we note the following easy consequence of Farkas' Lemma.

\begin{lemma}\label{lem:Farkas-truc}
  The following are equivalent:
  \begin{enumerate}\renewcommand{\theenumi}{\roman{enumi}}
  \item The half-line $D(\id)+\RR_+D(\chi^U)$ defined by the pair
  $\pair{\id}{U}$ is an edge of $\close{Q_n}$. 
  \item There exists a matrix $D$ satisfying the following constraints:
    \begin{subequations} \label{eq:edge-cond-plain}
      \begin{align}
        \label{eq:idIsOpt}
        D\bullet D(\pi) &> D\bullet D(\id) && \forall\; \pi \not= \id,\id^-,\\
        \label{eq:cutIsOpt}
        D\bullet D(\chi^{U'}) &> D\bullet D(\chi^U)=0 && \forall\; U' \not= U,\cplmt U.
      \end{align}
    \end{subequations}
  \item There exists a matrix $C$ satisfying
    \begin{subequations}  \label{eq:edge-cond-Farkas}
      \begin{alignat}{3}
        \label{eq:idIsOneOpt}
        &C\bullet D(\pi)  &&\ge C\bullet D(\id) &\qquad& \forall\; \pi \not= \id,\id^-,\\
        \label{eq:cutsArePositive}
        &C\bullet D(\chi^{U'}) &&\ge 0               &      & \forall\; U' \not= U,\cplmt U,\\
        \label{eq:cutIsNegative}
        &C\bullet D(\chi^U)    &&< 0. &&
      \end{alignat}
    \end{subequations}
  \end{enumerate}
\end{lemma}

Condition~\eqref{eq:edge-cond-plain} is easier to check for individual matrices, but
condition~\eqref{eq:edge-cond-Farkas} will be needed in a proof below.

We move on to the next example which both provides some cases needed for the proof of
Theorem~\ref{thm:conjectureAboutExraysProven} and motivates the following definitions.

Let $U$ be a subset of $[n]$ and consider its representation as a word of
length $n$.  We say that a maximal sequence of consecutive $0$s in this word
is a {\it valley} of $U$. In other words, a valley is 
an inclusion wise maximal subset $[l,l+j] \subset \cplmt U$.
Accordingly, a maximal sequence of consecutive $1$s is called a {\it hill}.
A valley and a hill meet at a {\it slope}.  Thus the number of slopes is the
number of occurrences of the patterns $01$ and $10$ in the word, or in other
words, the number of $k\in[n-1]$ with $k\in U$ and 
$k+1\not\in U$ or vice versa.  If all valleys and hills of a subset $U$ of
$[n]$ consist of only one 
element (as for example in $10101$) or, equivalently, if $U$ has the maximal
possible number $n-1$ of 
slopes, or, equivalently, if $U$ consists of all odd or all even numbers in
$[n]$, we speak of an {\it 
  alternating} set.

\begin{lemma}\label{lem:ShownAllreadyForClosureOfQn} 
  For every set $\{W_1,\dots,W_r\}$ of non-empty
  proper subsets of $[n]$ incident on $\pi$, there is a matrix $C$ 
  such that the minimum $C\bullet D(\sigma)$ over all $\sigma\in S(n)$ is
  attained solely in $\pi$ and 
  $\pi^-$, and that $C\bullet D(\chi^{U'}) \ge 0$ for every non-empty proper
  subset $U'$ of $[n]$ where equality 
  holds precisely for the sets $W_i$ and their complements.
  This implies that $D(\pi)+\cone\{D(\chi^{W_1}),\dots,D(\chi^{W_r})\}$ is a
  face of the polyhedron $\close{Q_n}=P_n+$CUT$_n$.
\end{lemma}
\begin{proof}
Follows from Proposition~\ref{prop:n-1Dimensional}.
\end{proof}

\begin{example}[Unbounded edges of $\close{Q_4}$]\label{ex:Q4}
  We consider the edges of $\close{Q_4}$ containing $D(\id)=D(\id^-)$ (this is
  justified by 
  Remark~\ref{rem:symm-Qn-and-inciovrdg}).  We distinguish the sets $U$ by
  their number of slopes. 
  Clearly, a set $U$ with a single slope is incident either to $\id$ or to
  $\id^-$, and we have already 
  dealt with that case in Lemma~\ref{lem:ShownAllreadyForClosureOfQn}.
  The following sets have two slopes: $0100$, $0110$, $0010$, $1011$, $1001$,
  and $1101$. 
  We only have to consider $1011$, $1001$, and $1101$, because the others are
  their complements. 
  The first one, $1011$, is almost incident $\id^-$, and the last one,
  $1101$, is almost incident to 
  $\id$, so we know that the pairs $\pair{\id}{1011}$ and $\pair{\id}{1101}$
  do not define edges of $\close{Q_4}$ by
  Lemma~\ref{lem:theyAreNoExtremeRays}. 
  For the remaining set with two slopes, $1001$, the following matrix satisfies
  property~\eqref{eq:edge-cond-Farkas} with $C$ replaced by $\Cm{1001}$ and
  $U$ by $1001$: 
  \begin{equation*}
    \Cm{1001}:=
    \Mtxv{
      \hphantom{-}0 & \hphantom{-}1 & -2 & \hphantom{-}1\\
      \hphantom{-}1 & \hphantom{-}0 & \hphantom{-}3 & -2\\
      -2  & \hphantom{-}3 & \hphantom{-}0 & \hphantom{-}1\\
      \hphantom{-}1 & -2  & \hphantom{-}1 & \hphantom{-}0
    }.
  \end{equation*}
  The two alternating sets (i.$\,$e., sets with tree slopes) are $1010$ and
  $0101$, which are almost incident to $\id$ and $\id^-$ respectively.
  This concludes the discussion of $\close{Q_4}$.
\end{example}


Having settled some of the cases for small values of $n$, we give the result
by which the reduction to 
smaller $n$ is performed, which is an important ingredient for settling
Theorem~\ref{thm:conjectureAboutExraysProven}.  The following lemma shows that
unbounded edges of $\close{Q_n}$ can be ``lifted'' to a larger polyhedron
$\close{Q_{n+k}}$. 

\begin{lemma} \label{lem:reduct}
  Let $U_0$ be a non-empty proper subset of $[n]$ whose word has the form
  $a1b$ for two (possibly empty) 
  words $a,b$.  For any $k\ge 0$ define the subset $U_k$ of $[n+k]$ by its word
  \begin{equation*}
    U_k := a\ubtxt{1\dots1}{6mm}{$k+1$}b.
  \end{equation*}
  If the pair $\pair{\idn{n}}{U_0}$ defines an edge of $\close{Q_n}$, then the
  pair 
  $\pair{\idn{n+k}}{U_k}$ defines an edge of $\close{Q_{n+k}}$.
\end{lemma}
Note that the lemma also applies to consecutive zeroes, by exchanging the
respective set by its complement.
\begin{proof}
  Let $C\in \SzM{n}$ be a matrix satisfying conditions
  \eqref{eq:edge-cond-Farkas} for $U:=U_0$.  Fix 
  $k\ge 1$ and let $n':=n+k$.  We will construct a matrix $C'\in \SzM{n'}$
  satisfying 
  \eqref{eq:edge-cond-Farkas} for $U:=U_k$.
  For a ``big'' real number $\omega\ge 1$ define a matrix $B_\omega\in
  \SzM{k+1}$ whose entries are zero 
  except for those connecting $j$ and $j+1$, for $j\in[k]$:
  \begin{equation*}
    B_\omega :=
    \Mtxv{
      0     &\omega&      &       &      &      &      \\
      \omega&0     &\omega&       & \N0  &      &      \\
      ~     &\omega&      &       &      &      &      \\
      ~     &      &      &\ddots &      &      &      \\
      ~     &      &      &       &      &\omega&      \\
      ~     &\N0   &      &       &\omega&0     &\omega\\
      ~     &      &      &       &      &\omega&     0
    }%
    .
  \end{equation*}
  We use this matrix to put a heavy weight on the ``path'' which we ``contract.''
  For our second ingredient, let $l_a$ denote the length of the word $a$ and
  $l_b$ the length of the word 
  $b$ (note that $l_a=0$ and $l_b=0$ are possible).  Then we define
  \begin{align*}
    B_- &:=
    \Mtxv{
      +1           & \dots & +1         \\
      ~            &       &            \\
      \Zero_{k-1}  & \dots & \Zero_{k-1}\\ 
      ~            &       &            \\
      -1           & \dots & -1         \\ 
    } \in \MM((k+1)\times l_a)
    \qquad\text{and}\\
    B_+ &:=
    \Mtxv{
      -1           & \dots & -1         \\
      ~            &       &            \\
      \Zero_{k-1}  & \dots & \Zero_{k-1}\\ 
      ~            &       &            \\
      +1           & \dots & +1         \\ 
    } \in \MM((k+1)\times l_b)
    ,
  \end{align*}
  where $\Zero_{k-1}$ stands for a column of $k-1$ zeros.  Putting these
  matrices together we obtain an 
  $n'\times n'$-matrix $B$:
  \begin{equation*}
    B :=
    \begin{pmatrix}
      \N0     & B_-^\Tp  & \N0 \\
      B_-     & B_\omega & B_+ \\
      \N0     & B_+^\Tp  & \N0
    \end{pmatrix}.
  \end{equation*}
  
  Now it is easy to check that for any
  $\pi'\in \pi[n']$ we have $B\bullet D(\pi') \ge B\bullet D(\id)$.
  Moreover let $\pi'\in \pi[n']$ satisfy 
  $B\bullet D(\pi') < B\bullet D(\id) +1$.
  By exchanging $\pi'$ with $\pi'^-$, we can assume that $\pi'(1) <
  \pi'(n')$. 
  It is easy to see that such a $\pi'$ then has the following ``coarse
  structure''  
  \begin{equation} \label{eq:lemreduct:coarse-struct}
    \begin{aligned}
      \pi'([l_a])                  &\subset [l_a] \\
      \pi'([n']\setminus[n'-l_b])  &\subset [n']\setminus[n'-l_b] \\
      \pi'(j)                      &= j \quad \forall\; j\in \{l_a+1,\dots, l_a+k+1\}.
    \end{aligned}
  \end{equation}
  Thus the matrix $B$ enforces that the ``coarse structure'' of a $\pi'\in
  \pi[n']$ minimizing $B\bullet 
  D(\pi')$ coincides with $\id$.  We now modify the matrix $C$ to take care of
  the ``fine structure''. 
  For this, we split $C$ into matrices
  $C_{11}\in \SzM{l_a}$, %
  $C_{22}\in \SzM{l_b}$, %
  $C_{12}\in\MM(l_a\times l_b)$, %
  $C_{21}=C_{12}^\Tp\in\MM(l_b\times l_a)$, %
  and vectors
  $c\in\RR^{l_a}$, %
  $d\in\RR^{l_b}$ %
  as follows:
  \begin{equation*}
    C =
    \begin{pmatrix}
      C_{11} & c & C_{12} \\
      c^\Tp  & 0 & d^\Tp  \\
      C_{21} & d & C_{22} 
    \end{pmatrix}.
  \end{equation*}
  Then we define the ``stretched'' matrix $\check C\in\SzM{n'}$ by
\newcommand{\sz}{\scriptstyle0}%
  \begin{equation*}
    \check C :=
    \begin{pmatrix}
      C_{11}   & c   && \M0       &&\Zero& C_{12}     \\
      c^\Tp    &\sz  &&           && \sz & \Zero^\Tp  \\
      ~        &     &&           &&     &            \\
      \M0      &     && \M0       &&     & \M0        \\
      ~        &     &&           &&     &            \\
      \Zero^\Tp& \sz &&           &&\sz  & d^\Tp      \\
      C_{21}   &\Zero&& \M0       && d   & C_{22}
    \end{pmatrix}
  \end{equation*}
  where the middle $\M0$ has dimensions $(k-1)\times(k-1)$.  Finally we let
  $C' := B +\eps \check C$, 
  where $\eps>0$ is small.  We show that $C'$ satisfies \eqref{eq:edge-cond-Farkas}.
  
  We first consider $C'\bullet D(\chi^{U'})$ for non-empty subsets
  $U'\subsetneq[n']$.  Note that, if 
  $U'$ contains $\{l_a+1,\dots,l_a+k+1\}$, then for $U:=
  U'\setminus\{l_a+1,\dots,l_a+k+1\}$, we have 
  $C'\bullet D(\chi^{U'}) = C\bullet D(\chi^U)$.  Thus we have $C'\bullet
  D(\chi^{U_k}) = C\bullet 
  D(\chi^{U_0}) < 0$ proving \eqref{eq:cutIsNegative} for $C'$ and $U_k$.  For
  every other $U'$ with 
  $C'\bullet D(\chi^{U'}) <0$, if $\omega$ is big enough, then either $U'$ or
  $\cplmt U'$ contains 
  $\{l_a+1,\dots,l_a+k+1\}$, and w.l.o.g.\ we assume that $U'$ does.  By
  \eqref{eq:cutsArePositive} 
  applied to $C$ and $U$, we know that this implies $U=U_0$ or $U=\cplmt U_0$
  and hence $U'=U_k$ or 
  $\cplmt U'=U_k$.  Thus, \eqref{eq:cutsArePositive} holds for $C'$ and $U_k$. 
  
  Second, we address the permutations.  To show \eqref{eq:idIsOneOpt}, let
  $\pi'\in S(n)$ be given which 
  minimizes $C'\bullet D(\pi')$.  Again, by replacing $\pi'$ by $\pi'^-$ if
  necessary, we assume 
  $\pi'(1)<\pi'(n')$ w.l.o.g.  If $\eps$ is small enough, we know that
  $\pi'$ has the coarse structure 
  displayed in \eqref{eq:lemreduct:coarse-struct}.  This implies that we can
  define a permutation $\pi\in S(n)$ by letting 
  \begin{equation*}
    \pi(j) :=
    \begin{cases}
      \pi'(j)        & \text{if } j\in [l_a],\\
      \pi'(j)=j      & \text{if } j=l_a+1,\\
      \pi'(j-k) + k  & \text{if } j\in [n]\setminus[l_a+1].
    \end{cases}
  \end{equation*}
  An easy but lengthy computation (see \cite{S09} for the details) shows that
 \begin{align*}
    C'\bullet D(\pi') - C'\bullet D(\idn{n'})&\geq\eps\Big[ C\bullet D(\pi) + k\cdot C\bullet\Mtxv{\M0_{l_a\times
        l_a} & \M1\\ \M1 & \M0_{l_b\times l_b}} \\
     &-\lt(
      C\bullet D(\idn{n})
      + k\cdot C\bullet\Mtxv{\M0_{l_a\times l_a} & \M1\\ \M1 & \M0_{l_b\times l_b}}
      \rt )\Big]\\
    &= \eps \big[ C\bullet D(\pi) - C\bullet D(\idn{n})\big] \ge 0.
  \end{align*}
Thus \eqref{eq:idIsOneOpt} holds.
\end{proof}

\begin{example}
  We give an example for the application of Lemma~\ref{lem:reduct}.
  For $n=5$, consider the half-line defined by the pair $\pair{\id}{11001}$.
  The set $11001$ can be 
  reduced to $1001$ by contracting the hill $1\path 2$.  To do so we set
  \begin{equation*}
    \Cm{11001}:=
    \eps \Mtxv{
      \hphantom{-}0 & \hphantom{-}0 & \hphantom{-}0 & \hphantom{-}0 & \hphantom{-}0\\
      \hphantom{-}0 & \hphantom{-}0 & \hphantom{-}1 & -2 & \hphantom{-}1\\
      \hphantom{-}0 & \hphantom{-}1 & \hphantom{-}0 & \hphantom{-}3 & -2\\
      \hphantom{-}0 & -2 & \hphantom{-}3 & \hphantom{-}0 & \hphantom{-}1\\
      \hphantom{-}0 & \hphantom{-}1 & -2 & \hphantom{-}1 & \hphantom{-}0
    }
    +
    \Mtxv{
      \hphantom{-}0 & \hphantom{-}\omega & -1 & -1 & -1\\
      \hphantom{-}\omega & \hphantom{-}0 & \hphantom{-}1 & \hphantom{-}1 & \hphantom{-}1\\
      -1 & \hphantom{-}1 & \hphantom{-}0 & \hphantom{-}0 & \hphantom{-}0\\
      -1 & \hphantom{-}1 & \hphantom{-}0 & \hphantom{-}0 & \hphantom{-}0\\
      -1 & \hphantom{-}1 & \hphantom{-}0 & \hphantom{-}0 & \hphantom{-}0
    }
  \end{equation*}
  for a small $\eps>0$ and a big $\omega\ge 1$.
\end{example}


After these preparations we can tackle the proof of the theorem.

\begin{proof}[Proof of Theorem~\ref{thm:conjectureAboutExraysProven}]
  By Remark~\ref{rem:symm-Qn-and-inciovrdg}, we only need to consider
  $\pi=\id$.  We distinguish the sets 
  $U$ by their numbers of slopes.

  \paragraph{\it One slope.}
  This is equivalent to $U$ or $\cplmt U$ being incident to $\id$. We have
  treated this case in Lemma~\ref{lem:ShownAllreadyForClosureOfQn}.

  \paragraph{\it Two slopes.}
  The complete list of all possibilities, up to complements, and how they are
  dealt with is summarized in 
  Table~\ref{table:2slopes}.  In this table, $0$ stands for a valley
  consisting of a single zero while 
  $0\dots 0$ stands for a valley consisting of at least two zeros (the same
  with hills).  The matrices 
  for the reduced words satisfying \eqref{eq:edge-cond-Farkas} can be found in the appendix on
  page~\pageref{table:appendix}. The condition \eqref{eq:edge-cond-Farkas} can
  be verified by some case distinctions.

  \begin{table}[thb] 
    \begin{center}
      \caption{List of all sets with two slopes (up to complement)}
      \begin{tabular}{ccccl}
        \toprule  
        \addlinespace[1ex]
        \multicolumn{3}{c}{Word} & Edge? & Why?\\
        \addlinespace[1ex]
        Hill 1 & Valley & Hill 2 \\
        \addlinespace[1ex]\rowcolor[gray]{0.9}  
        1 & 0 & 1 &no& almost incident to $\id$\\
        \addlinespace[1ex]
        1 & 0 & $1\dots1$ &no& almost incident to $\id^-$ \\
        \addlinespace[1ex]\rowcolor[gray]{0.9}
        1 & $0\dots0$ & 1 &yes& reduce to $n=4$, $1001$, by Lemma~\ref{lem:reduct}\\
        \addlinespace[1ex]
        1 & $0\dots0$ & $1\dots1$ &yes& reduce to $n=4$, $1001$, by Lemma~\ref{lem:reduct}\\
        \addlinespace[1ex]\rowcolor[gray]{0.9}
        $1\dots1$ & 0 & 1 &no& almost incident to $\id$ \\
        \addlinespace[1ex]
        $1\dots1$ & 0 & $1\dots1$ &yes& reduce to $n=5$, $11011$, by Lemma~\ref{lem:reduct}\\
        \addlinespace[1ex]\rowcolor[gray]{0.9}
        $1\dots1$ & $0\dots0$ & 1 &yes& reduce to $n=4$, $1001$, by Lemma~\ref{lem:reduct}\\
        \addlinespace[1ex]
        $1\dots1$ & $0\dots0$ & $1\dots1$ &yes& reduce to $n=5$, $11011$, by Lemma~\ref{lem:reduct} \\   
        \addlinespace[1ex]
        \toprule
      \end{tabular}
      \label{table:2slopes}
    \end{center}
  \end{table}

  \paragraph{\it Three slopes.}
  This case can be tackled using the same methods we applied in the case above.
  Table~\ref{table:3slopes} gives the results.

  \begin{table}[thb] 
    \begin{center}
      \caption{List of all sets with three slopes (up to complement)}
      \begin{tabular}{cccccl}
        \toprule  
        \addlinespace[1ex]
        \multicolumn{4}{c}{Word} & Edge? & Why?\\
        \addlinespace[1ex]
        Hill 1&Valley 1&Hill 2&Valley 2\\
        \addlinespace[1ex]\rowcolor[gray]{0.9}  
        1 & 0 & 1 & 0 &no& almost incident to $\id$\\
        \addlinespace[1ex]
        1 & 0 & 1 & $0\dots0$ &no& almost incident to $\id$\\
        \addlinespace[1ex]\rowcolor[gray]{0.9}
        1 & 0 & $1\dots1$ & 0 &yes& reduce to $n=5$, $10110$, by Lemma~\ref{lem:reduct}\\
        \addlinespace[1ex]
        1 & 0 & $1\dots1$ & $0\dots0$ &yes&  reduce to $n=5$, $10110$, by Lemma~\ref{lem:reduct} \\   
        \addlinespace[1ex]\rowcolor[gray]{0.9}
        1 & $0\dots0$ & 1 & 0 &yes& reduce to $n=5$, $10010$, by Lemma~\ref{lem:reduct}\\
        \addlinespace[1ex]
        1 & $0\dots0$ & 1 & $0\dots0$ &yes&  reduce to $n=5$, $10010$, by Lemma~\ref{lem:reduct}\\
        \addlinespace[1ex]\rowcolor[gray]{0.9}
        1 & $0\dots0$ & $1\dots1$ & 0 &yes&  reduce to $n=5$, $10010$, by Lemma~\ref{lem:reduct} \\   
        \addlinespace[1ex]
        1 & $0\dots0$ & $1\dots1$ & $0\dots0$ &yes&  reduce to $n=5$, $10110$, by Lemma~\ref{lem:reduct}\\
        \addlinespace[1ex]\rowcolor[gray]{0.9}
        $1\dots1$ & 0 & 1 & 0&no& almost incident to $\id$ \\
        \addlinespace[1ex]
        $1\dots1$ & 0 & 1 & $0\dots0$ &no& almost incident to $\id$\\
        \addlinespace[1ex]\rowcolor[gray]{0.9}
        $1\dots1$ & 0 & $1\dots1$ & 0 &yes&  reduce to $n=5$, $10110$, by Lemma~\ref{lem:reduct}\\
        \addlinespace[1ex]
        $1\dots1$ & 0 & $1\dots1$ & $0\dots0$ &yes&  reduce to $n=5$, $10110$, by Lemma~\ref{lem:reduct}\\
        \addlinespace[1ex]
        $1\dots1$ & $0\dots0$ & 1 & 0 &yes&  reduce to $n=5$, $10010$, by Lemma~\ref{lem:reduct}\\
        \addlinespace[1ex]\rowcolor[gray]{0.9}
        $1\dots1$ & $0\dots0$ &1 & $0\dots0$ &yes&  reduce to $n=5$, $10010$, by Lemma~\ref{lem:reduct}\\
        \addlinespace[1ex]
        $1\dots1$ & $0\dots0$ & $1\dots1$ & 0 &yes&  reduce to $n=5$, $10010$, by Lemma~\ref{lem:reduct}\\
        \addlinespace[1ex]\rowcolor[gray]{0.9}
        $1\dots1$ & $0\dots0$ & $1\dots1$ & $0\dots0$ &yes&  reduce to $n=5$, $10010$, by
        Lemma~\ref{lem:reduct}\\
        \addlinespace[1ex]
        \toprule
      \end{tabular}
      \label{table:3slopes}
    \end{center}
  \end{table}
  
  \paragraph{\it $s\ge 4$ slopes.}
  
  Using Lemma~\ref{lem:reduct}, we reduce such a set to an alternating set
  with $s$ slopes showing that 
  for all these sets $U$ the pair $\pair{\id}{U}$ defines an edge of
  $\close{Q_n}$.  This is in accordance with the statement of the theorem
  because sets which are almost incident to $\id$ can have at most three
  slopes. 
  The statement for alternating sets is proven by induction on $n$ in
  Lemma~\ref{lem:induction} below. 
  Note that the starts of the inductions in the proof of that lemma are $n=5$
  and $n=6$ for even or odd $s$ respectively.
  
  This concludes the proof of the theorem.
\end{proof}

We now present the inductive construction which we need for the case of an
even number of $s\ge 4$ slopes.

\begin{lemma} \label{lem:induction}
  For an integer $n\ge 5$ let $U$ be an alternating subset of $[n]$.  The pair
  $\pair{\id}{U}$ defines an edge of $\close{Q_n}$.
\end{lemma}
\begin{proof}
  We first prove the case when $n$ is odd.

  The proof is by induction over $n$. For the start of the induction we
  consider $n=5$ and offer the 
  matrix $\Cm{10101} \in\SzM{5}$ in Table~\ref{table:appendix} of the
  appendix satisfying \eqref{eq:edge-cond-plain}.  
  We will need this matrix in the inductive construction.
  
  Now set $E^5 := \Cm{10101}$ and assume that the pair $\pair{\id}{U^-}$
  defines an edge of 
  $\close{Q_n}$ where $U^-$ is an alternating subset 
  of $[n]$.  W.l.o.g., we assume that $U^-=10\dots01$.  There exists a matrix
  $E^-\in\SzM{n}$ for which 
  \eqref{eq:edge-cond-plain} holds.  We will construct a matrix
  $E\in\SzM{n+2}$ satisfying 
  \eqref{eq:edge-cond-plain} for $U := 010\dots010$.
  
  We extend $E^-$ to a $(n+2)\times(n+2)$-Matrix
  \begin{equation*}
    \widehat E := 
    \Mtxv{
      \textstyle E^-&& \Zero &\Zero \\
      ~           &&       &     \\
      \Zero^\Tp   && 0     &0     \\
      \Zero^\Tp   && 0     &0     \\
    }.
  \end{equation*}
  We do the same with $E^5$, except on the other side:
  \begin{equation*}
    \widehat{E^5} :=
    \Mtxv{
      0     &0     &&\Zero^\Tp   \\
      0     &0     &&\Zero^\Tp   \\
      ~     &      &&            \\
      \Zero &\Zero &&\textstyle E^5\\
      }.
  \end{equation*}
  Now we let $E:= \widehat E + \widehat{E^5}$ and check the conditions
  \eqref{eq:edge-cond-plain} on $E$.  These are now easily verified.

  For the even case we guarantee the start of induction investigating
  $n=6$. We give a matrix 
  $\Cm{101010}$ satisfying \eqref{eq:edge-cond-plain} in
  Table~\ref{table:appendix} in the appendix. 
  (Note that $101010$ is the only set which is not incident to $\id$, 
  is not almost incident to $\id$ or $\id^-$, cannot be reduced by
  Lemma~\ref{lem:reduct} and is no complement of sets of any of these three types.)  
  The induction is proved in the same way by using the matrix $E^6:=\Cm{101010}$.
\end{proof}


%% file: Appendix.tex
\section*{Appendix}

\begin{table}[thb] 
\begin{center}
\caption{\label{table:appendix} Matrices certifying unbounded edges of $Q_n$}
 \begin{tabular}{ccrl}
   \toprule  
   \addlinespace[1ex]
   $n$ &Slopes&\multicolumn{2}{c}{Matrix}\\
   \addlinespace[1ex]\rowcolor[gray]{0.9}  
   \addlinespace[1ex]\rowcolor[gray]{0.9}
   4 &2& $\Cm{1001} :=$&$
   \begin{pmatrix}
     \;0 & \;1 & -2 & \;1\\
     \;1 & \;0 & \;3 & -2\\
     -2  & \;3 & \;0 & \;1\\
     \;1 & -2  & \;1 & \;0
   \end{pmatrix}$\\
   \addlinespace[1ex]
   5 &2& $\Cm{11011}:=$&$
    \begin{pmatrix}
      \;0 & \;8 & -6 & -1 & -1\\
      \;8 & \;0 & \;2 & \;9 & -3\\
      -6 & \;2 & \;0 & \;5 & -7\\
      -1 & \;9 & \;5 & \;0 & 11\\
      -1 & -3 & -7 & 11 & \;0
    \end{pmatrix}$\\
   \addlinespace[1ex]\rowcolor[gray]{0.9}
   5 &3& $\Cm{10110}:=$&$
   \begin{pmatrix}
      \;0 & \;2 & \;2 & \;1 & -3\\
      \;2 & \;0 & \;0 & -2 & \;2\\
      -2 & \;0 & \;0 & \;2 & \;0\\
      \;1 & -2 & \;2 & \;0 & \;1\\
      -3 & \;2 & \;0 & \;1 & \;0
    \end{pmatrix}$\\
   \addlinespace[1ex]
    5  &3& $\Cm{10010}:=$&$
    \begin{pmatrix}
      \;0 & \;2 & -2 & \;2 & -2\\
      \;2 & \;0 & \;4 & -3 & \;1\\
      -2 & \;4 & \;0 & \;1 & \;1\\
      \;2 & -3 & \;1 & \;0 & \;1\\
      -2 & \;1 & \;1 & \;1 & \;0
    \end{pmatrix}$\\
   \addlinespace[1ex]\rowcolor[gray]{0.9}
   5 &4& $\Cm{10101}:=$&$
   \begin{pmatrix} 
   \;0 & \;0 & \;3 & -2 & -1\\
   \;0 & \;0 & \;1 & \;1 & -2\\
   \;3 & \;1 & \;0 & \;1 & \;3\\
   -2 & \;1 & \;1 & \;0 & \;0\\
   -1 & -2 & \;3 & \;0 & \;0
   \end{pmatrix}$\\
   \addlinespace[1ex]
    6 &5& $\Cm{101010}:=$&$
    \begin{pmatrix} 
      \;0 & \;0 & \;1 & -1 & \;0 & \;0\\
      \;0 & \;0 & \;1 & \;1 & -2 & \;0\\
      \;1 & \;1 & \;0 & \;1 & \;3 & -2\\
      -1 & \;1 & \;1 & \;0 & \;0 & \;1\\
      \;0 & -2 & \;3 & \;0 & \;0 & \;1\\
      \;0 & \;0 & -2 & \;1 & \;1 & \;0
    \end{pmatrix}$\\
   \addlinespace[1ex]
   \toprule
 \end{tabular}
\end{center}
\end{table}
